\documentclass{amsart}

\usepackage{amsfonts,amsmath,amsthm,amscd,amssymb,latexsym,cite,verbatim,texdraw,floatflt,pb-diagram}
\usepackage{tikz}

\newtheorem{theorem}{Theorem}[section]

\newtheorem{proposition}[theorem]{Proposition}
\newtheorem{corollary}[theorem]{Corollary}

\theoremstyle{definition}
\newtheorem{definition}[theorem]{Definition}
\newtheorem{example}[theorem]{Example}

\theoremstyle{remark}

\numberwithin{equation}{section}

\newcommand{\ve}{\varepsilon}

\newcommand{\NN}{\mathbb{N}}

\begin{document}

\title[{Periodic points of mappings contracting total pairwise distance}]{Periodic points of mappings contracting total pairwise distance}

\author[E. Petrov]{Evgeniy Petrov}

\address{Institute of Applied Mathematics and Mechanics of the NAS of Ukraine, Batiuka str. 19, 84116, Slovyansk, Ukraine}

\email{eugeniy.petrov@gmail.com}



\begin{abstract}
We consider a new type of mappings in metric spaces so-called mappings contracting total pairwise distance on $n$ points. It is shown that such mappings are continuous. A theorem on the existence of periodic points for such mappings is proved and the classical Banach fixed-point theorem is obtained like a simple corollary as well as the fixed point theorem for mappings contracting perimeters of triangles. Examples of mappings contracting total pairwise distance on $n$ points and having different properties are constructed.
\end{abstract}

\maketitle

\section{Introduction}

The Banach contraction principle has been generalized in many ways over the years. It is possible to distinguish at least three types of generalizations of this theorem: in the first case the contractive nature of the mapping is weakened, see, e.g.~\cite{BW69,Ki03,Mk69,Pr20,Ra62,Su06,Wa12,P23};
in the second case the topology is weakened, see, e.g.~\cite{Ba00, JS15,Fr00,HZ07,KKR90,Ta74,SIIR20}; 
the third case is multi-valued generalizations, see, e.g.~\cite{Na69,Ma68,AK72}. Such generalizations are very numerous and as a rule establish the existence and uniqueness of a fixed point. There are considerably fewer theorems which establish the existence of periodic points of mappings in general metric spaces. The most known is Edelstein's~\cite{Ed62} theorem stating that any $\ve$-contractive mapping of a nonempty compact metric space into itself has a periodic point. See also some generalizations of this theorem~\cite{Ra06,CJU08,Ba66,Se72}. Such theorems do not generalize Banach's theorem and have a completely different proof scheme. At the same time, we note that the study of periodic points play an important role in the theory of dynamical systems~\cite{De22} and is also well developed for spaces of special types~\cite{Sa64,Bo71,BGMY80,Fr92,AM65,SY17}.

In this paper we introduce a new type of mappings in general metric spaces, which we call mappings contracting total pairwise distance on $n$ points and prove the existence theorem of periodic points for such mappings. In contrast to many other periodic point theorems, this theorem is a proper generalization of Banach's contraction principle.

Everywhere below by $|X|$ we denote the cardinality of the set $X$.

Let $(X,d)$ be a metric space, $|X|\geqslant 2$, and let $x_1$, $x_2$, \ldots, $x_n \in X$, $n\geqslant 2$. Denote by
\begin{equation}\label{e0}
S(x_1,x_2,\ldots,x_n)=\sum\limits_{1\leqslant i<j\leqslant n}d(x_i,x_j)
\end{equation}
the sum of all pairwise distances between the points from the set $\{x_1, x_2, \ldots, x_n\}$, which we call \emph{total pairwise distance}.

\begin{definition}\label{d1}
Let $n\geqslant 2$ and let $(X,d)$ be a metric space with $|X|\geqslant n$. We shall say that $T\colon X\to X$ is a \emph{mapping contracting total pairwise distance on $n$ points} if there exists $\alpha\in [0,1)$ such that the inequality
  \begin{equation}\label{e1}
S(Tx_1,Tx_2,\ldots,Tx_n) \leqslant \alpha S(x_1,x_2,\ldots,x_n)
  \end{equation}
  holds for all $n$ pairwise distinct points $x_1, x_2, \ldots, x_n \in X$.
\end{definition}

Note that the requirement for $x_1, x_2, \ldots, x_n\in X$ to be pairwise distinct is essential, which is confirmed by the following proposition.

\begin{proposition}
Suppose that in Definition~\ref{d1} inequality~(\ref{e1}) holds for any $n$ points $x_1, x_2, \ldots, x_n\in X$ with $|\{x_1, x_2, \ldots, x_n\}|=k$, where $2\leqslant k\leqslant n-1$. Then $T$ is a mapping contracting total pairwise distance on $k$ points.
\end{proposition}
\begin{proof}
Let $\{x_1, x_2, \ldots, x_k\}$ be the set of pairwise distinct points from $X$. Consider $k$ sets $A_i=\{x_1, x_2, \ldots, x_k, x_{k+1},\ldots,x_n\}$ consisting of $n$ points such that
$x_{k+1}=\cdots=x_n=x_i$, $i=1,\ldots,k$, i.e., the element $x_i$ occurs $n-k+1$ times in the set $A_i$.
Hence, by the supposition the inequalities
$$
S(Tx_1, Tx_2, \ldots, Tx_k, Tx_i,\ldots,Tx_i)\leqslant \alpha S(x_1, x_2, \ldots, x_k, x_i,\ldots,x_i),
$$
hold for all  $i=1,\dots,k$.
Summarizing the left and right sides of these inequalities we get
$$
\sum\limits_{i=1}^k S(Tx_1, Tx_2, \ldots, Tx_k, Tx_i,\ldots,Tx_i)\leqslant \alpha \sum\limits_{i=1}^k S(x_1, x_2, \ldots, x_k, x_i,\ldots,x_i).
$$
Hence,
\begin{multline*}
\sum\limits_{i=1}^k
\bigg(
 S(Tx_1,\ldots, Tx_k,)+ (n-k)\sum\limits_{j=1}^k d(Tx_j,Tx_i)
+S(Tx_i,\ldots,Tx_i)
\bigg) \\
\leqslant \alpha \sum\limits_{i=1}^k
\bigg(
S(x_1, \ldots, x_k)+(n-k)\sum\limits_{j=1}^k d(x_j,x_i)
+S(x_i,\ldots,x_i)
\bigg).
\end{multline*}
It is easy to see that $S(Tx_i,\ldots,Tx_i)=S(x_i,\ldots,x_i)=0$,
$$
\sum\limits_{i=1}^k
\sum\limits_{j=1}^k d(Tx_j,Tx_i)= 2S(Tx_1,\ldots, Tx_k)
$$
and
$$
\sum\limits_{i=1}^k
\sum\limits_{j=1}^k d(x_j,x_i) = 2S(x_1, \ldots, x_k).
$$
Hence, we obtain
$$
(k+2(n-k))S(Tx_1,Tx_2,\ldots,Tx_k)\leqslant \alpha (k+2(n-k))S(x_1,x_2,\ldots,x_k).
$$
Reducing both parts by $k+2(n-k)$ we get the desired assertion.
\end{proof}

\begin{proposition}
Mapping contracting total pairwise distance on $m$ points, $m\geqslant 2$, is a mapping contracting total pairwise distance on $n$ points for all $n>m$.
\end{proposition}
\begin{proof}
It is clear that it is sufficient to prove this proposition for $n=m+1$. Let the points $x_1, x_2, \ldots, x_n \in X$ be  pairwise distinct. Since $m=n-1$ by the supposition the following inequalities hold:
$$
S(Tx_2,Tx_3,\ldots,Tx_n) \leqslant \alpha S(x_2,x_3,\ldots,x_n),
$$
$$
S(Tx_1,Tx_3,\ldots,Tx_n) \leqslant \alpha S(x_1,x_3,\ldots,x_n),
$$
$$\cdots$$
$$
S(Tx_1,Tx_2,\ldots,Tx_{n-1}) \leqslant \alpha S(x_1,x_2,\ldots,x_{n-1}).
$$
Summarizing the left and right sides of these inequalities and simplifying the obtained inequality we get inequality~(\ref{e1}).
\end{proof}

\section{The main results}

Recall that for a given metric space $X$, a point $x \in X$ is said to be an \emph{accumulation point} of  $X$ if every open ball centered at $x$ contains infinitely many points of $X$.

\begin{proposition}\label{p13}
Mappings contracting total pairwise distance on $n$ points are continuous.
\end{proposition}
\begin{proof}
If $n=2$, then this proposition states that Banach contractions are continuous, which is well-known. Let $(X,d)$ be a metric space with $|X|\geqslant n\geqslant 3$, $T\colon X\to X$ be a mapping contracting total pairwise distance on $n$ points and let $x_1$ be an isolated point in $X$. Then, clearly, $T$ is continuous at $x_1$. Let now $x_1$ be an accumulation point. Let us show that for every $\ve>0$, there exists $\delta>0$ such that $d(Tx_1,Tx_2)<\ve$ whenever $d(x_1,x_2)<\delta$.
Since $x_1$ is an accumulation point, for every $\delta>0$ there exists $x_2,\ldots,x_{n}\in X$ such that $d(x_1,x_i)< \delta$ for $i=2,\ldots,n$ and the points $x_1,x_2,\ldots,x_{n}$ are pairwise distinct.  By~(\ref{e0}) and~(\ref{e1}) we have
$$
d(Tx_1,Tx_2)\leqslant S(Tx_1, Tx_2,\ldots,Tx_{n})
$$
$$
\leqslant \alpha S(x_1, x_2,\ldots,x_{n}).
$$
Using the triangle inequality $d(x_i,x_j) \leqslant d(x_1,x_i)+d(x_1,x_j)$ and the inequalities $d(x_1,x_i)< \delta$ for $i=2,\ldots,n$, we have
$$
d(Tx_1,Tx_2)< \alpha k(n) \delta,
$$
where $k(n)\in \NN$ is a number of inequalities of the form $d(x_1,x_i)< \delta$ applied in this estimation. Setting $\delta=\ve /(k(n)\alpha)$, we obtain the desired inequality.
\end{proof}

Let $T$ be a mapping on the metric space $X$. A point $x\in X$ is called a \emph{periodic point of period $n$} if $T^n(x) = x$. The least positive integer $n$ for which $T^n(x) = x$ is called the prime period of $x$, see, e.g.,~\cite[p.~18]{De22}. Note that a fixed point is a point of prime period $1$.

The following theorem is the main result of this paper.

\begin{theorem}\label{t1}
Let $n\geqslant 2$, $(X,d)$ be a complete metric space with $|X|\geqslant n$ and let $T\colon X\to X$ be a mapping contracting total pairwise distance on $n$ points in $X$. Then $T$ has a periodic point of prime period $k$, $k\in \{1,\ldots,n-1\}$. The number of periodic points is at most $n-1$.
\end{theorem}
\begin{proof}
Suppose that $T$ does not have periodic points of prime period $k$, $k\in \{1,\ldots,n-1\}$. In particular, it means that $T$ has no fixed points. Let $x_0\in X$, $Tx_0=x_1$, $Tx_1=x_2$, \ldots, $Tx_n=x_{n+1}$, \ldots . Let us show that all $x_i$ are different. Since $x_i$ is not fixed, then $x_i\neq x_{i+1}=Tx_i$.
If $n=2$, then the proof should be read starting from the next paragraph.
Let $n=3$. Since $T$ has no periodic points of prime period $2$ we have $x_{i+2}=T(T(x_i))\neq x_i$ and by the supposition that $x_{i+1}$ is not fixed we have $x_{i+1}\neq x_{i+2}=Tx_{i+1}$. Hence, $x_i$, $x_{i+1}$ and $x_{i+2}$ are pairwise distinct.
Further, let $n\geqslant 4$. Since $T$ has no periodic points of prime period $3$ we have $x_{i+3}=T(T(Tx_i))\neq x_i$. Since $T$ has no periodic points of prime period $2$ we have $x_{i+3}=T(Tx_{i+1})\neq x_{i+1}$, and by the supposition that $x_{i+2}$ is not fixed we have $x_{i+2}\neq x_{i+3}=Tx_{i+2}$. Hence, $x_i$, $x_{i+1}$, $x_{i+2}$ and $x_{i+3}$ are pairwise distinct.
Repeating these considerations $n-4$ times we see that the points $x_i, x_{i+1}, x_{i+2}$, \ldots , $x_{i+n-1}$ are pairwise distinct for every $i=0,1,\ldots$.

Further, set
$$
s_0=S(x_0,x_1\ldots,x_{n-1}),
$$
$$
s_1=S(x_1,x_2,\ldots,x_n),
$$
$$
\cdots
$$
$$
s_n=S(x_n,x_{n+1},\ldots,x_{2n-1}),
$$
$$
\cdots .
$$
Since  $x_i, x_{i+1}, x_{i+2}$, \ldots , $x_{i+n-1}$ are pairwise distinct by~(\ref{e1}) we have $s_1\leqslant \alpha s_0$, $s_2\leqslant \alpha s_1$, \ldots, $s_n\leqslant \alpha s_{n-1}$ and
\begin{equation}\label{e2}
s_0>s_1>\ldots>s_n>\ldots .
\end{equation}
Suppose that $j\geqslant n$ is a minimal natural number such that $x_j=x_i$ for some $i$ such that $0\leqslant i<j-n+1$. Then  $x_{j+1}=x_{i+1}$, $x_{j+2}=x_{i+2}$,\ldots Hence, $s_i=s_j$ which contradicts to~(\ref{e2}).

Further, let us show that $(x_i)$ is a Cauchy sequence.  It is clear that
$$
d(x_0,x_1)\leqslant s_0,
$$

$$
d(x_1,x_2)\leqslant s_1\leqslant \alpha s_0,
$$

$$
d(x_2,x_3)\leqslant s_2\leqslant \alpha s_1\leqslant  \alpha^2 s_0,
$$
$$
\cdots
$$
$$
d(x_{n},x_{n+1})\leqslant s_{n}\leqslant \alpha^{n} s_0,
$$
$$
d(x_{n+1},x_{n+2})\leqslant s_{n+1}\leqslant \alpha^{n+1} s_0,
$$

$$
\cdots .
$$
By the triangle inequality,
$$
d(x_n,\,x_{n+p})\leqslant d(x_{n},\,x_{n+1})+d(x_{n+1},\,x_{n+2})+\ldots+d(x_{n+p-1},\,x_{n+p})
$$
$$
\leqslant \alpha^{n}s_0+\alpha^{n+1}s_0+\cdots +\alpha^{n+p-1}s_0 = \alpha^{n}(1+\alpha+\ldots+\alpha^{p-1})s_0
=\alpha^{n}\frac{1-\alpha^{p}}{1-\alpha}s_0.
$$
Since by the supposition $0\leqslant\alpha<1$, then $d(x_n,\,x_{n+p})<\alpha^{n}\frac{1}{1-\alpha}s_0$. Hence, $d(x_n,\,x_{n+p})\to 0$ as $n\to \infty$ for every $p>0$. Thus, $\{x_n\}$ is a Cauchy sequence. By the completeness of $(X,d)$, this sequence has a limit $x^*\in X$.

Let us prove that $Tx^*=x^*$. Since $x_n\to x^*$, and by Proposition~\ref{p13} the mapping $T$ is continuous, we have $x_{n+1}=T x_n\to Tx^*$.  By the triangle inequality we have
$$
d(x^*,Tx^*)\leqslant d(x^*,x_{n})+d(x_{n},Tx^*)
\to 0
$$
as $n\to \infty$, which means that $x^*$ is the fixed point and contradicts to our assumption.

Suppose that there exists $n$ pairwise distinct periodic points $x_1,\ldots,x_n$. Let $p(x_i)$ be a prime period of $x_i$, $i=1,\ldots,n$ and let $k$ be the least common multiple of the integers $p(x_1),\ldots,p(x_n)$. It is clear that
\begin{equation}\label{e5}
T^kx_i=x_i \, \text{ for all } \, i=1,\ldots,n.
\end{equation}
Further, observe the following: let $x$ and $y$ be different periodic points of $T$. If the orbits $O(x)$ and $O(y)$ are different sets, then, clearly, $T^jx\neq T^jy$ for all $j=1,2,\ldots$. If $O(x)=O(y)$, then the relation $T^jx\neq T^jy$ is also evident for all $j=1,2,\ldots$. Hence, we can assert that for all $j=1,2,\ldots$ the elements of the set $\{T^jx_1,T^jx_2,\ldots,T^jx_n\}$ are pairwise distinct and we can apply inequality~(\ref{e1}) consecutively $k$ times:
\begin{multline*}
S(T^kx_1,T^kx_2,\ldots,T^kx_n) \leqslant
\alpha S(T^{k-1}x_1,T^kx_2,\ldots,T^{k-1}x_n) \leqslant
 \ldots \\ \leqslant \alpha^{k-1} S(Tx_1,Tx_2,\ldots,Tx_n) \leqslant \alpha^k S(x_1,x_2,\ldots,x_n).
\end{multline*}
But by~(\ref{e5}) the equality
$$
S(T^kx_1,T^kx_2,\ldots,T^kx_n)=S(x_1,x_2,\ldots,x_n)
$$
holds, which is a contradiction.
\end{proof}

Let $(X,d)$ be a metric space. Then a mapping $T\colon X\to X$ is called a \emph{contraction mapping} on $X$ if there exists $\alpha\in [0,1)$ such that
\begin{equation}\label{e3}
d(Tx,Ty)\leqslant \alpha d(x,y)
\end{equation}
for all $x,y \in X$.

\begin{corollary}
(Banach fixed-point theorem)
Let $(X,d)$ be a  non\-empty complete metric space with a contraction mapping $T\colon X\to X$.  Then $T$ admits a unique fixed point.
\end{corollary}
\begin{proof}
If $|X|=1$, then Banach fixed point theorem is trivial. For $|X|\geqslant 2$ this theorem follows from Theorem~\ref{t1} by setting $n=2$. In this case inequality~(\ref{e1}) turns into~(\ref{e3}).
\end{proof}

The following definition was introduced in~\cite{P23}. In particular, it is a partial case of Definition~\ref{d1} when $n=3$.
\begin{definition}\label{ed1}
Let $(X,d)$ be a metric space with $|X|\geqslant 3$. We shall say that $T\colon X\to X$ is a \emph{mapping contracting perimeters of triangles} on $X$ if there exists $\alpha\in [0,1)$ such that the inequality
  \begin{equation}\label{ee1}
   d(Tx,Ty)+d(Ty,Tz)+d(Tx,Tz) \leqslant \alpha (d(x,y)+d(y,z)+d(x,z))
  \end{equation}
  holds for all three pairwise distinct points $x,y,z \in X$.
\end{definition}
The following statement was proved in~\cite[Theorem~2.4]{P23} and it is a direct consequence of Theorem~\ref{t1} in the case $n=3$.
\begin{corollary}\label{c1}
Let $(X,d)$, $|X|\geqslant 3$, be a complete metric space and let $T\colon X\to X$ be a mapping contracting perimeters of triangles on $X$. Then $T$ has a fixed point if and only if $T$ does not possess periodic points of prime period $2$. The number of fixed points is at most two.
\end{corollary}

\begin{example}
Let $X$ be a metric space with $|X|=n\geqslant 2$, $Y$ be a subset of $X$ with $1\leqslant |Y| \leqslant n-1$ and let $T\colon X\to Y$ be any mapping. It is clear that inequality~(\ref{e1}) holds, where $\{x_1,x_2,\ldots,x_n\}=X$. Hence, by Theorem~\ref{t1} the mapping $T$ has a periodic point of prime period $k\in \{1,\ldots,n-1\}$.
Note that there can be no periodic points in the set $X\setminus Y$.
In particular, it follows from this example that any mapping $T\colon Y\to Y$ has a periodic point, where $Y$ is a finite nonempty metric space.
\end{example}

\begin{example}
Let us construct an example of a mapping contracting perimeters of triangles and having two periodic points on a metric space $X$ with $|X|=\mathfrak{c}$.
Let $(X,d)$ be a metric space such that $X=\{0\}\cup\{1\}\cup[3,\infty)\subseteq \mathbb R^1$, where $d$ is the usual Euclidean distance. Define a mapping $T$ as follows:
$$
Tx=
\begin{cases}
1, &\text{if } x=0,\\
0, &\text{if } x=1,\\
1, &\text{if } x\in [3,\infty).
\end{cases}
$$
It is clear that $0$ and $1$ are periodic points of prime period $2$. The simple proof of the fact that $T$ is a mapping contracting perimeters of triangles is left to the reader.
\end{example}

\begin{proposition}
Suppose that under the supposition of Theorem~\ref{t1} the mapping $T$ has a fixed point $x^*$ which is a limit of some iteration sequence $x_0, x_1=Tx_0, x_2=Tx_1,\ldots$ such that $x_i\neq x^*$ for all $i=1,2,\ldots$. Then $x^*$ is the unique fixed point.
\end{proposition}
\begin{proof}
For the case $n=3$ this assertion was shown in~\cite{P23}. Let now $n\geqslant 4$. Indeed, suppose that $T$ has another fixed point $x^{**}\neq x^*$. It is clear that $x_i\neq x^{**}$ for all $i=1,2,\ldots$. Moreover, it is clear that $x_i\neq x_j$ for $i\neq j$. Otherwise this sequence is cyclic starting from some index and can not be convergent to $x^*$. Hence, we have that the points $x^*$, $x^{**}$ and $x_i$, $x_j$ are pairwise distinct for all $i,j=1,2,\ldots$, $i\neq j$.

Consider the ratio
$$
R_i=\frac{S(Tx^*,Tx^{**},Tx_i,Tx_{i+1},\ldots,Tx_{i+n-3})}
{S(x^*,x^{**},x_i,x_{i+1},\ldots,x_{i+n-3})}
$$
$$
=\frac{S(x^*,x^{**},x_{i+1},x_{i+2},\ldots,x_{i+n-2})}
{S(x^*,x^{**},x_i,x_{i+1},\ldots,x_{i+n-3})}.
$$
Taking into consideration that $d(x^*,x_{i})\to 0$ and $d(x^{**},x_{i})\to d(x^{**},x^*)$, we obtain
$R_i\to 1$ as $i\to \infty$, which contradicts to condition~(\ref{e1}).
\end{proof}

\begin{proposition}\label{p1}
Let $n\geqslant 2$, $(X,d)$ be a metric space, $|X|\geqslant n$, and let  $T\colon X\to X$ be a mapping contracting total pairwise distance on
$n$ points. If $x$ is an accumulation point of  $X$, then inequality~(\ref{e3}) holds for all points $y\in X$.
\end{proposition}
\begin{proof}
If $n=2$, then $T$ is a contraction mapping and inequality~(\ref{e3}) holds. The case $n=3$ was proved in~\cite[Proposition 2.7]{P23}. Suppose that $n\geqslant 4$. Let $x\in X$ be an accumulation point and let $y\in X$. If $y=x$, then clearly~(\ref{e3}) holds. Let now $y\neq x$. Since $x$ is an accumulation point, then there exists a sequence $z_n\to x$ such that $z_n \neq x$, $z_n \neq y$ and all $z_n$ are different.
Hence, by~(\ref{e1}) the inequality
\begin{equation*}
  S(Tx,Ty,Tz_i,\ldots,Tz_{i+n-3})\leqslant \alpha S(x,y,z_i,\ldots,z_{i+n-3})
\end{equation*}
holds for every $i\in \mathbb N$. Since $d(x,z_i)\to 0$ and every metric is continuous we have $d(z_i,y) \to d(x,y)$. Since $T$ is continuous, we have $d(Tx,Tz_i)\to 0$ and, consequently,  $d(Tz_i,Ty)\to d(Tx,Ty)$. Observing that $|{z_i,\ldots,z_{i+n-3}}|=n-2$ and letting $i\to \infty$, we obtain
  \begin{equation*}
   d(Tx,Ty)+(n-2)d(Tx,Ty) \leqslant \alpha (d(x,y)+(n-2)d(x,y)),
  \end{equation*}
which is equivalent to~(\ref{e3}).
\end{proof}

\begin{corollary}\label{cor1}
Let $n \geqslant 2$, $(X,d)$ be a metric space, $|X|\geqslant n$, and let $T\colon X\to X$ be a mapping contracting total pairwise distance on $n$ points. If all points of $X$ are accumulation points, then $T$ is a contraction mapping.
\end{corollary}

\begin{example}
Let us construct an example of a mapping $T\colon X\to X$ contracting total pairwise distance on $n$ points, $n\geqslant 3$, but not contracting total pairwise distance on $n-1$ points for a metric space $X$ with $|X|=\aleph_0$. Let
$$
X=\{x^*, x_1^1,\ldots,x_1^{n-1},\ldots,x_{i}^1,\ldots,x_i^{n-1},\ldots \}
$$
and let $a$ and $\ve$ be positive real numbers. Define the metric $d$ on $X$ as follows:
$$
d(x,y)=
\begin{cases}
\frac{\ve}{2^{i-1}}, &\text{if } x=x_{i}^l, y=x_{i}^m, l\neq m,\\
&\text{and } i \text{ is odd},\\
\frac{\ve}{2^{i-2}}, &\text{if } x=x_{i}^l, y=x_{i}^m, l\neq m,\\
&\text{and } i \text{ is even},\\
\frac{a}{2^{i-1}}, &\text{if } x=x_i^l, \,  y=x_{i+1}^m,\\
d(x_i^1,x_{i+1}^1)+\cdots+d(x_{j-1}^1,x_j^1), &\text{if } x=x_i^l, \,  y=x_{j}^m ~~ i+1<j,\\
2a-d(x_1^1,x_i^l), &\text{if } x=x_i^l, \, y=x^*,\\
0, &\text{if } x=y,\\

\end{cases}
$$
where $1\leqslant l,m\leqslant n-1$, $i=1,2,\ldots$.

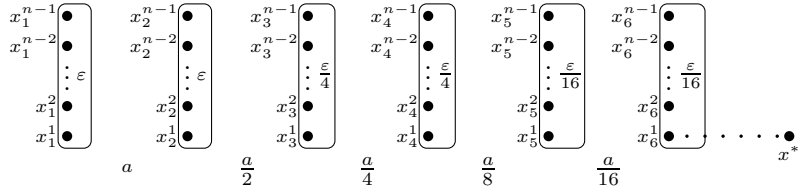
\begin{figure}[ht]
\begin{center}
\begin{tikzpicture}[scale=0.8]
\draw (1,-0.1-0.2) node [below] {{$\scriptstyle{a}$}};
\draw (3,-0.2) node [below] {{$\frac{a}{2}$}};
\draw (5,-0.2) node [below] {{$\frac{a}{4}$}};
\draw (7,-0.2) node [below] {{$\frac{a}{8}$}};
\draw (9,-0.2) node [below] {{$\frac{a}{16}$}};

\draw       (0,1) node [right] {$\scriptstyle{\ve}$};
\draw       (2,1) node [right] {$\scriptstyle{\ve}$};
\draw       (4,1) node [right] {\small{$\frac{\ve}{4}$}};
\draw       (6,1) node [right] {\small{$\frac{\ve}{4}$}};
\draw       (8,1) node [right] {\small{$\frac{\ve}{16}$}};
\draw       (10,1) node [right] {\small{$\frac{\ve}{16}$}};


\draw      (0,0) node [left] {$\scriptstyle{x_1^1}$};
\draw      (2,0) node [left] {$\scriptstyle{x_2^1}$};
\draw      (4,0) node [left] {$\scriptstyle{x_3^1}$};
\draw      (6,0) node [left] {$\scriptstyle{x_4^1}$};
\draw      (8,0) node [left] {$\scriptstyle{x_5^1}$};
\draw      (10,0) node [left] {$\scriptstyle{x_6^1}$};
\draw (12,0) node [below] {$\scriptstyle{x^*}$};

 \foreach \i in {(0,0),(2,0),(4,0),(6,0),(8,0),(10,0),(12,0)}
  \fill[black] \i circle (2.4pt);


\draw       (0,0.5) node [left] {$\scriptstyle{x_1^2}$};
\draw       (2,0.5) node [left] {$\scriptstyle{x_2^2}$};
\draw       (4,0.5) node [left] {$\scriptstyle{x_3^2}$};
\draw       (6,0.5) node [left] {$\scriptstyle{x_4^2}$};
\draw       (8,0.5) node [left] {$\scriptstyle{x_5^2}$};
\draw       (10,0.5) node [left] {$\scriptstyle{x_6^2}$};
 \foreach \i in {(0,0.5),(2,0.5),(4,0.5),(6,0.5),(8,0.5),(10,0.5)}
  \fill[black] \i circle (2.4pt);


\draw       (0,1.1) node  {$\vdots$};
\draw       (2,1.1) node  {$\vdots$};
\draw       (4,1.1) node  {$\vdots$};
\draw       (6,1.1) node  {$\vdots$};
\draw       (8,1.1) node  {$\vdots$};
\draw       (10,1.1) node  {$\vdots$};


\draw       (0,1.5) node [left] {$\scriptstyle{x_1^{n-2}}$};
\draw       (2,1.5) node [left] {$\scriptstyle{x_2^{n-2}}$};
\draw       (4,1.5) node [left] {$\scriptstyle{x_3^{n-2}}$};
\draw       (6,1.5) node [left] {$\scriptstyle{x_4^{n-2}}$};
\draw       (8,1.5) node [left] {$\scriptstyle{x_5^{n-2}}$};
\draw       (10,1.5) node [left] {$\scriptstyle{x_6^{n-2}}$};
 \foreach \i in {(0,1.5),(2,1.5),(4,1.5),(6,1.5),(8,1.5),(10,1.5)}
  \fill[black] \i circle (2.4pt);x


\draw       (0,2) node [left] {$\scriptstyle{x_1^{n-1}}$};
\draw       (2,2) node [left] {$\scriptstyle{x_2^{n-1}}$};
\draw       (4,2) node [left] {$\scriptstyle{x_3^{n-1}}$};
\draw       (6,2) node [left] {$\scriptstyle{x_4^{n-1}}$};
\draw       (8,2) node [left] {$\scriptstyle{x_5^{n-1}}$};
\draw       (10,2) node [left] {$\scriptstyle{x_6^{n-1}}$};

 \foreach \i in {(0,2),(2,2),(4,2),(6,2),(8,2),(10,2)}
  \fill[black] \i circle (2.4pt);

 \foreach \i in {(10+0.3,0),(10+0.6,0),(10+0.9,0),(10+1.2,0),(10+1.5,0),(10+1.8,0)}
  \fill[black] \i circle (0.8pt);

\draw[rounded corners=.1cm] (-0.15, -0.2) rectangle (0.4, 2.2) {};
\draw[rounded corners=.1cm] (2-0.15, -0.2) rectangle (2+0.4, 2.2) {};
\draw[rounded corners=.1cm] (4-0.15, -0.2) rectangle (4+0.45, 2.2) {};
\draw[rounded corners=.1cm] (6-0.15, -0.2) rectangle (6+0.45, 2.2) {};
\draw[rounded corners=.1cm] (8-0.15, -0.2) rectangle (8+0.6, 2.2) {};
\draw[rounded corners=.1cm] (10-0.15, -0.2) rectangle (10+0.6, 2.2) {};

\end{tikzpicture}
\begin{center}
\caption{The points of the space $(X,d)$ with distances between them.}\label{fig1}
\end{center}
\end{center}
\end{figure}
The distances between the points of the space $X$ are depicted in Figure~\ref{fig1}. In other words, all the pairwise distances between the points from the set $\{x_{i}^1,\ldots,x_i^{n-1}\}$ are equal. The distances between any point from the set $\{x_{i}^1,\ldots,x_i^{n-1}\}$ and any point from the set $\{x_{i+1}^1,\ldots,x_{i+1}^{n-1}\}$ are equal to $a/2^{i-1}$. The set of points $\{x^*,x_1^1,x_2^1,\ldots,x_i^1,\ldots\}$ is embeddable into the real line and $d(x_1^1,x^*)=2a$. The reader can easily verify that for sufficiently small $\ve$ and sufficiently large $a$ the metric $d$ is well-defined. Moreover, the space is complete with the single accumulation point $x^*$.

Define a mapping  $T\colon X\to X$ as $Tx_i^j=x_{i+1}^j$, for all $i=1,2,\ldots$, $j=1,\ldots,n-1$ and $Tx^*=x^*$. As usual, by $\tbinom{n}{k}$ we denote the number of combinations of $k$ elements from $n$-element set.
Since $$
S(x_i^1,\ldots, x_i^{n-1})=S(Tx_i^1,\ldots, Tx_i^{n-1})=S(x_{i+1}^1,\ldots, x_{i+1}^{n-1})= {\tbinom{n-1}{2}} \ve/2^{i-1}
$$
if $i$ is odd, by~(\ref{e1}) we have that $T$ is  not a mapping contracting total pairwise distance on $n-1$ points.

Let us show that $T$ is a mapping contracting total pairwise distance on $n$ points. Let $x_1,\ldots,x_n$ be pairwise distinct points from the space $X$. By the pigeonhole principle  there exists al least two different sets $\{x_{i}^1,\ldots,x_i^{n-1}\}$ and $\{x_{j}^1,\ldots,x_j^{n-1}\}$, $i\neq j$, such that each of them contains at least one element from the set $\{x_1,\ldots,x_n\}$. This means that among the summands of the total pairwise distance $S(x_1,\ldots,x_n)$ there are values depending on $a$. The summands depending on $\ve$  may be absent. Let
$$
S(x_1,\ldots,x_n) = s(\ve)+s(a),
$$
where $s(\ve)$ and $s(a)$ are the sums of summands depending on $\ve$ and $a$, respectively. Note the following simple fact: if $x$ and $y$ belong to different sets $\{x_{i}^1,\ldots,x_i^{n-1}\}$ and $\{x_{j}^1,\ldots,x_j^{n-1}\}$, $i\neq j$, then $d(Tx,Ty)=d(x,y)/2$.
Hence,
$$
S(Tx_1,\ldots,Tx_n) = s'(\ve)+s(a)/2,
$$
where $s'(\ve)$ is a corresponding sum depending on $\ve$. It is clear that $s'(\ve)\leqslant s(\ve)$. 
Consider the ratio
$$
\frac{S(Tx_1,\ldots,Tx_n)}{S(x_1,\ldots,x_n)}=
\frac{s'(\ve)+s(a)/2}{s(\ve)+s(a)}\leqslant \frac{s(\ve)+s(a)/2}{s(a)}
=\frac{s(\ve)}{s(a)}+\frac{1}{2}.
$$
Hence, in order to show inequality~(\ref{e1}) we have to show that for some $0<\beta<\frac{1}{2}$ the inequality $s(\ve)/s(a)\leqslant \beta$ holds for any set $\{x_1,\ldots,x_n\}$ of $n$ pairwise distinct points. Let $\{x_1,\ldots,x_n\}$ be any set of $n$ pairwise distinct points from $X$ and let $i$ be the smallest lower index of the point $x_i^j\in \{x_1,\ldots,x_n\}$.
It is not hard to see that under this condition in the case if $i$ is even the value $s(\ve)/s(a)$ is maximal if $\{x_1,\ldots,x_n\}=\{x_i^1,\ldots,x_i^{n-1},x_{i+1}^l\}$.
In the case if $i$ is odd the same maximum of $s(\ve)/s(a)$ is attained also if $\{x_1,\ldots,x_n\}=\{x_{i}^m,x_{i+1}^1,\ldots,x_{i+1}^{n-1}\}$.
Indeed, any other arrangement of points preserving this condition reduces the number of summands depending on $\ve$ and increases the number of summands depending on $a$ in the sum $S(x_1,\ldots,x_n)$.
If $i$ is odd, then
$$
s(\ve)/s(a)=(\tbinom{n-1}{2}\ve/2^{i-1})/((n-1)a/2^{i-1})=\tbinom{n-1}{2}\ve/(a(n-1))
$$
and if $i$ is even, then
$$
s(\ve)/s(a)=(\tbinom{n-1}{2}\ve/2^{i-2})/((n-1)a/2^{i-1})=2\tbinom{n-1}{2}\ve/(a(n-1)).
$$
Hence, for every $0<\beta<\frac{1}{2}$  there exist sufficiently small $\ve$ and sufficiently large $a$ such that the inequality $s(\ve)/s(a)\leqslant \beta$ holds. Thus, inequality~(\ref{e1}) can be satisfied for any $\frac{1}{2}<\alpha<1$ and any set of pairwise distinct points $\{x_1,\ldots.,x_n\}$  under the appropriate choice of $\ve$ and $a$.
\end{example}

\end{document}